\documentclass[12pt]{amsart}
\usepackage{amsmath,amssymb}
\usepackage[pagebackref=true,
pdftitle = {Frieze\ patterns\ as\ root\ posets\ and\ affine\ triangulations},
pdfkeywords = {frieze\ pattern; arrangement\ of\ hyperplanes; simplicial; Weyl\ groupoid},
pdfsubject = {52C35; 20F55},
pdfauthor = {Michael\ Cuntz}]
{hyperref}
\usepackage{hyperref}
\usepackage{longtable}
\usepackage{epsfig}
\usepackage{setspace}

\newtheorem{propo}{Proposition}[section]
\newtheorem{corol}[propo]{Corollary}
\newtheorem{theor}[propo]{Theorem}
\newtheorem{lemma}[propo]{Lemma}
\newtheorem{conje}[propo]{Conjecture}

\theoremstyle{definition}
\newtheorem{defin}[propo]{Definition}
\newtheorem{examp}[propo]{Example}

\theoremstyle{remark}
\newtheorem{remar}[propo]{Remark}

\numberwithin{equation}{section}

\newcommand{\NN }{\mathbb{N}}

\newcommand{\RR }{\mathbb{R}}

\newcommand{\ZZ }{\mathbb{Z}}

\newcommand{\id }{\mathrm{id}}

\newcommand{\Ac }{\mathcal{A}}

\newcommand{\Kc }{\mathcal{K}}

\newcommand{\imr}{\alpha_0}

\newcommand{\cAp }{\mathcal{E}}

\newcommand{\Rh }{\hat R_0}

\newcommand{\maxdense}{dense }
\newcommand{\maxdensen}{dense}

\title[Frieze patterns as root posets and affine triangulations]
{Frieze patterns as root posets\\ and affine triangulations}

\author{M.~Cuntz}
\address{Michael Cuntz,
Institut f\"ur Algebra, Zahlentheorie und Diskrete Mathematik,
Fakult\"at f\"ur Mathematik und Physik,
Leibniz Universit\"at Hannover,
Welfengarten 1,
D-30167 Hannover, Germany}
\email{cuntz@math.uni-hannover.de}

\begin{document}

\begin{abstract}
The entries of frieze patterns may be interpreted as coordinates of roots of a finite Weyl groupoid of rank two. We prove the existence of maximal elements in their root posets and classify those frieze patterns which can be used to build an affine simplicial arrangement.
\end{abstract}

\maketitle

\section{Introduction}

In this note we consider triangulations of convex polygons by non-intersecting diagonals, see Fig.\ \ref{etatriang} for five examples.
\begin{figure}
\begin{center}
\includegraphics[width=\textwidth]{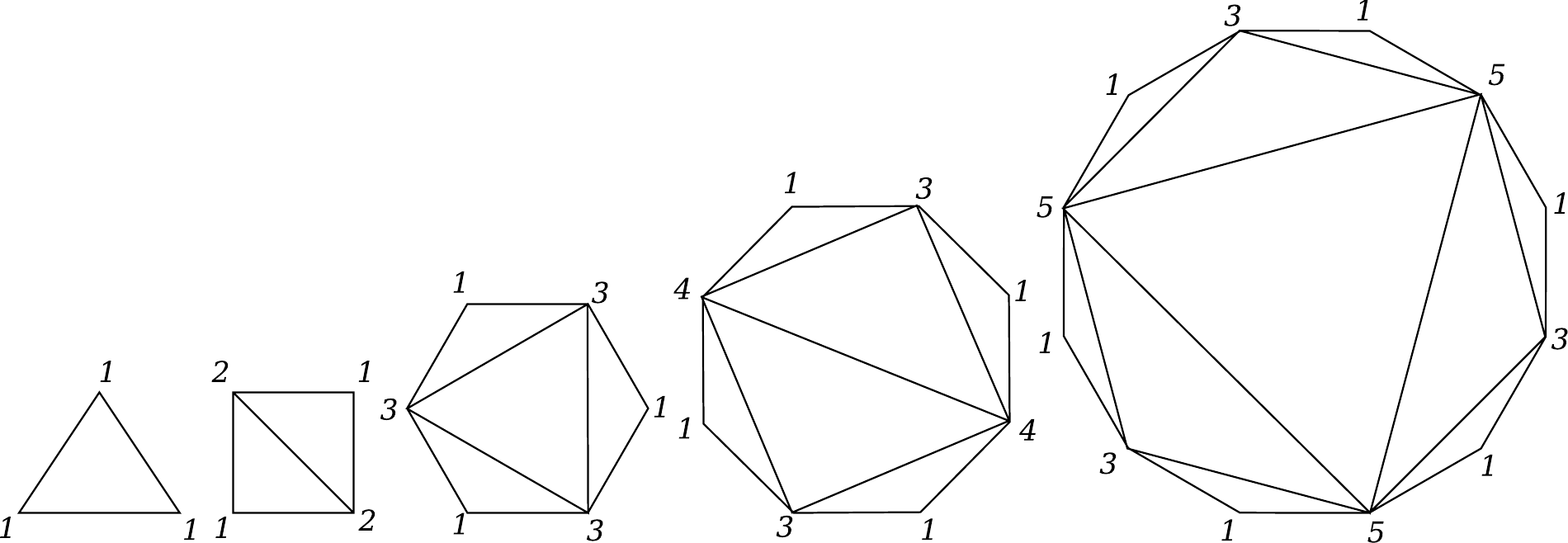}
\end{center}
\caption{Triangulations and $\eta$-sequences.\label{etatriang}}
\end{figure}
It is an easy exercise to show that such triangulations (combinatorially) correspond to what we call \emph{$\eta$-sequences}, see Def.\ \ref{eta_seq} or {\cite[Def.\ 3.2]{p-CH09d}} (these sequences are also called \emph{quiddity cycles} in \cite{jChC73}):
To each vertex we attach the number of triangles adjacent to this vertex and obtain a sequence $(c_1,c_2,\ldots,c_n)\in\NN^n$ (the numbers in Fig.\ \ref{etatriang}).
The number of $\eta$-sequences of length $(n+2)$ is the $n$-th Catalan number $C_n=\frac{1}{n+1}\binom{2n}{n}$. Indeed, $\eta$-sequences also correspond to Dyck words, expressions of parentheses, binary trees, noncrossing partitions, and there are many more constructions.

\begin{figure}
\begin{center}
\includegraphics[width=\textwidth]{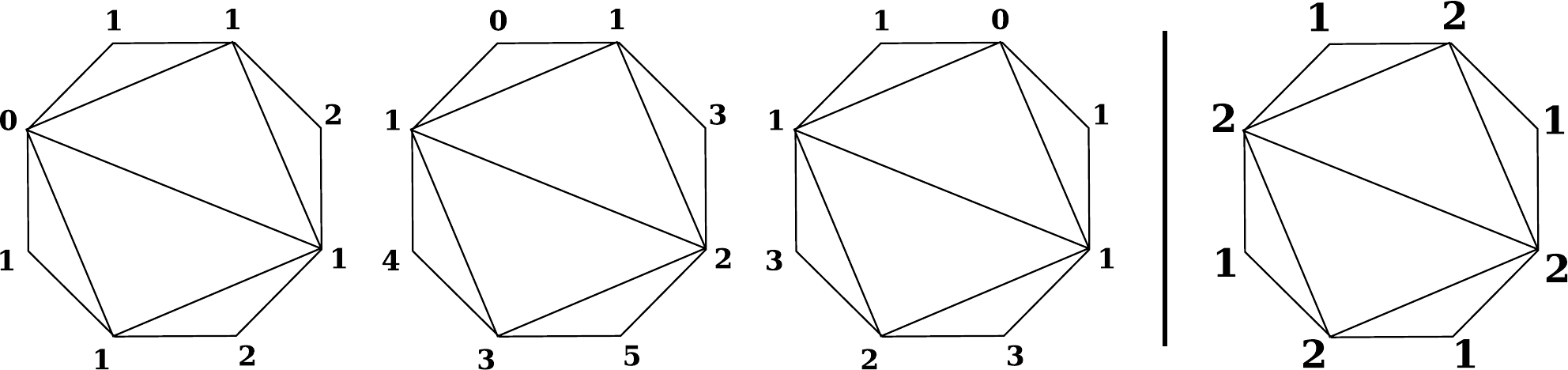}
\end{center}
\caption{Some values $\varphi_i(j)$ on the left, $|m_i|$ on the right.\label{phi_mis}}
\end{figure}

\begin{figure}
\begin{spacing}{0.6}
\[
\begin{array}{cccccccccccccccc}
 &0& &0& &0& &0& &0& &0& &0& &0\\
1& &1& &1& &1& &1& &1& &1& &1& \\
 &3& &1& &4& &1& &3& &1& &4& &1\\
2& &2& &3& &3& &2& &2& &3& &3& \\
 &1& &5& &2& &5& &1& &5& &2& &5\\
2& &2& &3& &3& &2& &2& &3& &3& \\
 &3& &1& &4& &1& &3& &1& &4& &1\\
1& &1& &1& &1& &1& &1& &1& &1& \\
 &0& &0& &0& &0& &0& &0& &0& &0
\end{array}
\]
\end{spacing}
\caption{The frieze pattern to the sequence $(3,1,4,1,3,1,4,1)$.\label{friezeex}}
\end{figure}
\begin{figure}
\begin{center}
\includegraphics[width=1.3\textwidth,clip=true,trim=100pt 200pt 0pt 180pt]{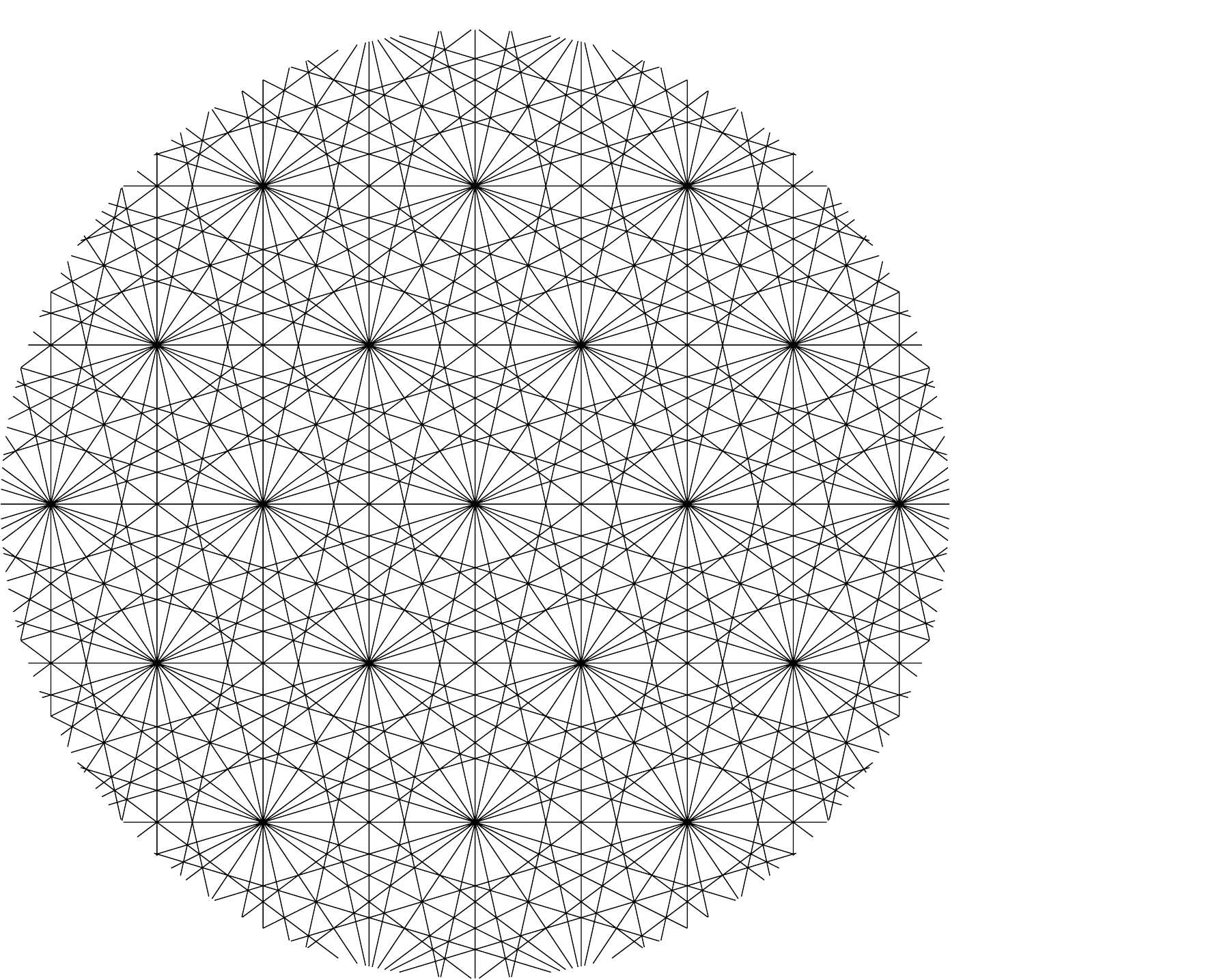}
\end{center}
\caption{An affine simplicial arrangement of imaginary type $A_1^{(1)}$ built upon the crystallographic arrangement with 
sequence $(1,3,1,5,1,3,1,5,1,3,1,5)$.\label{fig1315}}
\end{figure}

Now to each such triangulation and to each vertex $i$, attach numbers $\varphi_i(j)$, $j=1,\ldots,n$ in the following way (compare Fig.\ \ref{phi_mis} on the left side): Write $0$ at vertex $i$, $1$ at each adjacent vertex, and if two vertices of a triangle have labels $a$ and $b$ then write $a+b$ at the third vertex. The number written at vertex $j$ is $\varphi_i(j)$. The numbers $\varphi_i(j)$ are coordinates of roots of the associated \emph{Weyl groupoid} of rank two (see \cite{p-CH09d}), they also appeared earlier as entries of \emph{frieze patterns} (see for example \cite{jChC73} or \cite[p.\ 20]{MR1303141}), compare the Figures \ref{phi_mis} and \ref{friezeex}: The $\eta$-sequence is the second line of the frieze and the $\varphi_i(j)$ are all the entries.
Generalized frieze patterns have also been extensively studied recently in the context of cluster algebras, see for example \cite{pCfC06}.

For each $i$, let $m_i$ be the set of vertices $j$ such that $\varphi_i(j)$ is maximal among the numbers $\varphi_i(1),\ldots,\varphi_i(n)$. We call the triangulation \emph{dense} if there are no neighboring vertices $i,i+1$ such that $|m_i|=|m_{i+1}|=1$, where we identify $n+1$ and $1$ (see Fig.\ \ref{phi_mis} on the right side for an example for the numbers $|m_i|$).

The main result of this note is (Thm.\ \ref{cdense}):
\begin{theor}\label{thm1}
The dense triangulations are exactly the five triangulations displayed in Figure \ref{etatriang}.
\end{theor}

There are (at least) two applications for this result.

As for a Weyl group (see for example \cite{Arm2006}), one can associate a poset structure to the root systems of a Weyl groupoid. As an immediate corollary to Thm.\ \ref{cdense} we get that any finite Weyl groupoid of rank two has at least one object at which the root system has a unique maximal element (see Cor.\ \ref{maxelt_in_poset}).

The second application is a classification of the affine Weyl groupoids of rank three such that the parabolic subgroupoids containing an imaginary root are affine of type $A_1^{(1)}$ (see Thm.\ \ref{mainthm}). A complete classification of affine Weyl groupoids would be the first step in a classification of infinite dimensional Nichols algebras, see \cite{p-H-06} and \cite{p-AHS-08}.
We conjecture that the finite Weyl groupoids occurring as parabolic subgroupoids in affine Weyl groupoids of rank three are exactly the arithmetic root systems of rank two that appeared in the classification of finite dimensional Nichols algebras of diagonal type (see Conj.\ \ref{affwgconj}).

This note is organized as follows. Section \ref{densetriang} contains the proof of Thm.\ \ref{thm1}. In the last section we briefly recall the notions required for the mentioned applications and use Thm.\ \ref{thm1} to investigate root posets and affine Weyl groupoids.

\textbf{Acknowledgement.}
I am very grateful to the referees whose comments have helped me to improve the exposition considerably.

\section{Dense triangulations of polygons}\label{densetriang}

\begin{defin}
For $a\in\ZZ$, let
\[ \eta(a) := \begin{pmatrix} a & -1 \\ 1 & 0 \end{pmatrix}. \]
\end{defin}

\begin{remar}\label{eta_rule}
Notice the rule (compare \cite[Lemma 5.2]{p-CH09b})
\[ \eta(a)\eta(b) = \eta(a+1)\eta(1)\eta(b+1) \]
for all $a,b$.
\end{remar}

\begin{defin}[Compare {\cite[Def.\ 3.2]{p-CH09d}}]\label{eta_seq}
We define the set $\cAp$ of \emph{$\eta$-sequences} recursively by:
\begin{enumerate}
\item $(0,0)\in\cAp$.
\item If $(c_1,\ldots,c_n)\in\cAp$, then
$(c_2,c_3,\ldots,c_{n-1},c_n,c_1)\in\cAp$ and\\
$(c_n,c_{n-1},\ldots,c_2,c_1) \in\cAp$. \label{rgca_eta2}
\item If $(c_1,\ldots,c_n)\in\cAp$, then
$(c_1+1,1,c_2+1,c_3,\ldots,c_n)\in\cAp$. \label{rgca_eta3}
\end{enumerate}
\end{defin}

\begin{remar}\label{etatriangle}
If $(c_1,\ldots,c_n)\in\cAp$, then $\eta(c_1)\cdots\eta(c_n)=-\id$.
This is easy to see with the rule from Rem.\ \ref{eta_rule} which corresponds to Axiom \ref{rgca_eta3}.

An $\eta$-sequence of length $n$ may be visualized by a triangulation of a convex $n$-gon by non-intersecting diagonals (see Fig.\ \ref{etatriang}): An entry $c_i$ of the sequence corresponds to the $i$'th vertex; $c_i$ is the number of triangles ending at this vertex.
\end{remar}

\begin{defin}
Let $c=(c_1,\ldots,c_n)\in\cAp$.
An $1\le i\le n$ with $c_i=1$ is called an \emph{ear} of $c$.
\end{defin}

\begin{defin}
Let $(c_1,\ldots,c_n)\in\cAp$.
For a fixed $i\in\{1,\ldots,n\}$, let
\[ \varphi_i(j) :=
\begin{cases} 0 & j = i, \\
(\eta(c_{i})\cdots\eta(c_{j-1}))_{2,1} & j>i, \\
(\eta(c_{j})\cdots\eta(c_{i-1}))_{2,1} & j<i.
\end{cases} \]
\end{defin}

\begin{remar}\label{ci1}
Notice that $\varphi_i(j)=\varphi_j(i)$ for all $i,j$. Further, if $c_{i+1}=1$, then
\begin{eqnarray*}
\varphi_{i+1}(j) &=& \varphi_i(j)+\varphi_{i+2}(j),\\
\varphi_j(i+1) &=& \varphi_j(i)+\varphi_j(i+2)
\end{eqnarray*}
for all $j$.
Thus in the triangulation, a map $\varphi_i$ associates to each vertex a positive number in the following way (see the left three examples in Fig.\ \ref{phi_mis}, see also \cite[Section 3]{p-CH09d} for more details): Write a $0$ at vertex $i$ and $1$'s at all neighboring vertices. If $x,y,z$ are the vertices of a triangle and $x$, $y$ are labeled $v_x$, $v_y$ respectively, then $z$ is labeled $v_x+v_y$.
\end{remar}

\begin{lemma}\label{strictlyinc}
Let $c=(c_1,\ldots,c_n)\in\cAp$ and $1\le i<j\le n$. Assume that $c_\ell>1$ for all $\ell=i+1,\ldots,j-1$.
Then
\begin{eqnarray}
&&\varphi_i(i)<\varphi_i(i+1)<\ldots<\varphi_i(j-1)<\varphi_i(j),\\
&&\varphi_i(i)<\varphi_{i+1}(i)<\ldots<\varphi_{j-1}(i)<\varphi_j(i).
\end{eqnarray}
\end{lemma}
\begin{proof}
We proceed by induction: $\varphi_i(i)=0<1=\varphi_i(i+1)=\varphi_{i+1}(i)$, and for $\ell>i+1$,
\begin{eqnarray*}
&&\varphi_i(\ell)=c_{\ell-1} \varphi_i(\ell-1)-\varphi_i(\ell-2)>\varphi_i(\ell-1),\\
&&\varphi_\ell(i)=c_{\ell-1} \varphi_{\ell-1}(i)-\varphi_{\ell-2}(i)>\varphi_{\ell-1}(i),
\end{eqnarray*}
since $c_{\ell-1}\ge 2$.
\end{proof}

The following proposition will be the key to the main theorem of this section.

\begin{propo}\label{gggelll}
Let $c=(c_1,\ldots,c_n)\in\cAp$ and $1\le e_1<e_2\le n$ be ears.
Then there exists an $\ell$, $e_1\le\ell< e_2$ with
\begin{eqnarray}
\label{sat1} &&\varphi_j(e_1)<\varphi_j(e_2) \quad\text{for}\quad e_1\le j<\ell,\\
\label{sat2} &&\varphi_j(e_1)>\varphi_j(e_2) \quad\text{for}\quad \ell<j\le e_2,\\
\label{sat3} &&\varphi_\ell(e_1)\le \varphi_\ell(e_2).
\end{eqnarray}
\end{propo}
\begin{proof}
Assume first that $c_\ell>1$ for all $e_1<\ell<e_2$ (this means that there are no branches in the tree dual to the triangulation between $e_1$ and $e_2$).
Then by Lemma \ref{strictlyinc} and Rem.\ \ref{ci1},
\begin{eqnarray}
\label{e1e2_1}\varphi_{e_1}(e_1)<\varphi_{e_1+1}(e_1)<\ldots<\varphi_{e_2-1}(e_1)<\varphi_{e_2}(e_1),\\
\label{e1e2_2}\varphi_{e_1}(e_2)>\varphi_{e_1+1}(e_2)>\ldots>\varphi_{e_2-1}(e_2)>\varphi_{e_2}(e_2).
\end{eqnarray}
But $\varphi_{e_2}(e_1)>\varphi_{e_2}(e_2)=0=\varphi_{e_1}(e_1)<\varphi_{e_1}(e_2)$, thus there exists an $\ell$
satisfying (\ref{sat1}), (\ref{sat2}), and (\ref{sat3}).

Now to obtain the result for arbitrary $\eta$-sequences we may proceed by induction on the number of triangles which need to be attached to the shortest path of adjacent triangles from $e_1$ to $e_2$ in the triangulation corresponding to $c$ (this is unique because of the structure of a tree):
If we have a triangle at position $(i,i+1,i+2)$, i.e.\ $c_{i+1}=1$, then
\[ \varphi_{i+1}(j) = \varphi_i(j)+\varphi_{i+2}(j) \]
by Rem.\ \ref{ci1}.
Induction gives us an $\tilde\ell$ for the $\eta$-sequence without this triangle, i.e.\ removing the ear $i+1$ using the rule from Remark \ref{eta_rule}:
\begin{eqnarray*}
&&\varphi_j(e_1)<\varphi_j(e_2)\text{ for }i+1\ne j=e_1,\ldots,\tilde\ell-1,\\
&&\varphi_{\tilde\ell}(e_1)\le \varphi_{\tilde\ell}(e_2),\\
&&\varphi_j(e_1)>\varphi_j(e_2)\text{ for }i+1\ne j=\tilde{\ell}+1,\ldots,e_2.
\end{eqnarray*}
In the following two cases, $\ell=\tilde\ell$ is the wanted $\ell$ satisfying (\ref{sat1}), (\ref{sat2}), and (\ref{sat3}):\\
If $\varphi_i(e_1)<\varphi_i(e_2)$ and $\varphi_{i+2}(e_1)\le\varphi_{i+2}(e_2)$ then
\[ \varphi_{i+1}(e_1)=\varphi_i(e_1)+\varphi_{i+2}(e_1)<\varphi_i(e_2)+\varphi_{i+2}(e_2)=\varphi_{i+1}(e_2). \]
If $\varphi_i(e_1)\ge \varphi_i(e_2)$ and $\varphi_{i+2}(e_1)>\varphi_{i+2}(e_2)$ then
\[ \varphi_{i+1}(e_1)=\varphi_i(e_1)+\varphi_{i+2}(e_1)>\varphi_i(e_2)+\varphi_{i+2}(e_2)=\varphi_{i+1}(e_2). \]
Now if $\varphi_i(e_1)<\varphi_i(e_2)$ and $\varphi_{i+2}(e_1)>\varphi_{i+2}(e_2)$ then an adequate $\ell\in\{i,i+1\}$ exists as well:
If $\varphi_{i+1}(e_1)>\varphi_{i+1}(e_2)$, then choose $\ell=i$. If $\varphi_{i+1}(e_1)\le\varphi_{i+1}(e_2)$, then choose $\ell=i+1$.
\end{proof}

\begin{corol}\label{atmost2}
Let $c=(c_1,\ldots,c_n)\in\cAp$ and $1\le e_1<e_2\le n$ be ears. Then
\[ |\{ i\in\{1,\ldots,n\} \mid \varphi_i(e_1)=\varphi_i(e_2)\}| \le 2. \]
Moreover, if there are $i<j$ with $\varphi_i(e_1)=\varphi_i(e_2)$ and $\varphi_j(e_1)=\varphi_j(e_2)$, then either $i<e_1<j$ or $i<e_2<j$, so $j-i>1$.
\end{corol}
\begin{proof}
By Prop.\ \ref{gggelll} there is at most one $i$ with $\varphi_i(e_1)=\varphi_i(e_2)$ in each of the sets $\{i\mid e_1<i<e_2\}$ and $\{i\mid e_2<i \text{ or } i<e_1\}$. Further, $\varphi_{e_1}(e_1)\ne\varphi_{e_1}(e_2)$ and $\varphi_{e_2}(e_1)\ne\varphi_{e_2}(e_2)$.
\end{proof}

\begin{defin}\label{defmis}
Let $(c_1,\ldots,c_n)\in\cAp$. For a fixed $i\in\{1,\ldots,n\}$, let
\[ m_i := \{ j\in\{1,\ldots,n\} \mid \varphi_i(j)\ge \varphi_i(\ell) \text{ for all } \ell=1,\ldots,n \}. \]
\end{defin}

\begin{examp}
The right triangulation in Fig.\ \ref{phi_mis} is an example for the numbers $|m_i|$.
\end{examp}

\begin{defin}\label{deffan}
Let $c=(c_1,\ldots,c_n)\in\cAp$. Call $c$ \emph{fan-shaped} if
up to rotations $c=(n-2,1,2,\ldots,2,1)$.
\end{defin}

\begin{lemma}\label{lonlyears}
Let $c=(c_1,\ldots,c_n)\in\cAp$, $i\in\{1,\ldots,n\}$, and $j\in m_i$. Then either $j$ is an ear
or $c$ is fan-shaped.
\end{lemma}
\begin{proof}
Let $j\in m_i$. Then $\varphi_i(j)\ge \varphi_i(j+1)$ and $\varphi_i(j)\ge \varphi_i(j-1)$. Remember further that
\begin{equation}\label{j1j1}
\varphi_i(j+1)+\varphi_i(j-1) = c_j \varphi_i(j).
\end{equation}
Assume that $j$ is not an ear, thus $c_j>1$. If $c_j>2$ then $\varphi_i(j+1)>2\varphi_i(j)\ge 2\varphi_i(j+1)$ which is impossible. So $c_j=2$. But then $\varphi_i(j)=\varphi_i(j+1)=\varphi_i(j-1)$ by Equation \ref{j1j1}. Thus $j-1,j+1\in m_i$. Using induction we obtain a subsequence $(1,2,\ldots,2,1)$ around the position $j$ (possibly going beyond positions $1$ or $n$). This is only possible if $c$ is of the claimed form.
\end{proof}

\begin{defin}\label{defdense}
Let $c=(c_1,\ldots,c_n)\in\cAp$. Call $c$ \emph{\maxdensen} if
for all $i=1,\ldots,n$ either $|m_i|>1$ or $|m_{i+1}|>1$, where $m_{n+1}:=m_1$.
\end{defin}
\begin{remar}\label{onlyears}
It is easy to check that if $c$ is fan-shaped and
\maxdense of length $n$, then $n\in \{3,4\}$.
Thus if $n>4$ and if $c=(c_1,\ldots,c_n)\in\cAp$ is an arbitrary \maxdense sequence, then every $m_i$ only consists of ears by Lemma \ref{lonlyears}.
\end{remar}

\begin{propo}\label{earinmi}
Let $c=(c_1,\ldots,c_n)\in\cAp$. If $k$ is the number of ears in $c$, then
\[ |\{ i\in\{1,\ldots,n\} \mid |m_i|>1 \}| \le k. \]
\end{propo}
\begin{proof}
The case of fan-shaped sequences is easy to check. Thus assume that $c$ is not fan-shaped and therefore that all $m_i$ only consist of ears.
For ears $e,f$ let $$J_{e,f}:=\{i\in\{1,\ldots,n\}\mid \varphi_i(e)\ge \varphi_i(f)\}.$$ By Prop.\ \ref{gggelll}, these sets $J_{e,f}$ are of the form $\{\ell,\ell+1,\ldots,j\}$ for some $\ell\le j$ (view the $\eta$-sequence as a cycle, so possibly $J_{e,f}=\{1,\ldots,j,\ell,\ldots,n\}$ and $j<\ell$). Thus the intersection $$I_e:=\bigcap_{f \text{ ear}} J_{e,f}$$ also has the form $\{\ell,\ldots,j\}$, $\ell\le j$ because $e\notin J_{e,f}$ for any ear $f\ne e$. Notice that $e\in m_i$ if and only if $i\in I_e$: If $\varphi_i(e)\ge\varphi_i(f)$ for each ear $f$, then $\varphi_i(e)\ge\varphi_i(\ell)$ for each vertex $\ell$ by the same argument as in the proof of Lemma \ref{lonlyears}.

Now we count the number of $i$ with $|m_i|>1$: Let $k_1$ be the number of ears $f$ with $|I_f|=1$, and $k_2$ be the number of ears $f$ with $|I_f|>1$; we have $k\ge k_1+k_2$. Further, let
\begin{eqnarray*}
N_1 &:=& \{ i\in\{1,\ldots,n\} \mid |m_i|>1, \:\:\exists\: f\in m_i \::\: |I_f|=1\},\\
N_2 &:=& \{ i\in\{1,\ldots,n\} \mid |m_i|>1, \:\:\forall\: f\in m_i \::\: |I_f|>1\}.
\end{eqnarray*}
Then clearly $|N_1|\le k_1$. Let $i\in N_2$ and $e,f\in m_i$, $e\ne f$. Then $i\in I_e\cap I_f$. By Cor.~\ref{atmost2}, $|I_e\cap I_f|\le 2$ and $I_e\cap I_f$ consists of elements on the ``borders'' $\partial I_e$ and $\partial I_f$ of the intervals $I_e$ and $I_f$.
So $i\in N_2$ and $f\in m_i$ imply $i\in \partial I_f$. Thus the number of pairs $(i,f)$ with $i\in N_2$, $f\in m_i$ is at most twice the number of ears $f$ with $|I_f|>1$. We obtain
\begin{equation}\label{zweifach}
2|N_2| \le \sum_{i\in N_2} |m_i| \le \sum_{f \text{ ear},\: |I_f|>1} 2 = 2 k_2.
\end{equation}
Thus $|\{ i\in\{1,\ldots,n\} \mid |m_i|>1 \}| = |N_1|+|N_2| \le k_1+k_2 \le k$.
\end{proof}

\begin{corol}\label{cform}
Let $c=(c_1,\ldots,c_n)\in\cAp$ be \maxdensen.
\begin{enumerate}
\item Either $n=3$ and $c=(1,1,1)$, or $n$ is even and $c$ is of the form
\[ (1,*,1,*,\ldots) \quad\text{or}\quad (*,1,*,1,\ldots). \]
\item If $n>4$ then $(|m_1|,|m_2|,\ldots)\in \{(1,2,1,2,\ldots),(2,1,2,1,\ldots)\}$.
\item If $e$ is an ear, then there exists an $i$ with $e\in m_i$.
\end{enumerate}
\end{corol}
\begin{proof}
(1) Let $k$ be the number of ears.
Notice first that $k\le \frac{n}{2}$ except for $n=3$ in which case $c=(1,1,1)$ and $|m_1|=|m_2|=|m_3|=2$.
Assume now that $n>3$. Then by Prop.\ \ref{earinmi}, the number $d$ of $i$ with $|m_i|>1$ is at most $k$. But if $c$ is \maxdensen, then at least $\frac{n}{2}$ labels $i$ satisfy $|m_i|>1$. Thus $\frac{n}{2}\le k\le \frac{n}{2}$ if $n>3$. This is only possible if $n$ is even and $k=\frac{n}{2}$, hence every second entry in $c$ is a $1$. Notice also that $d=k$.

We now prove (2). We assume that $n>4$, so $c$ is not fan-shaped because it is \maxdensen.
The equality $d=k$ implies
\begin{equation}\label{m1m2}
(|m_1|,|m_2|,\ldots)\in \{(1,*,1,*,\ldots),(*,1,*,1,\ldots)\}
\end{equation}
where $*$ always denotes an integer greater than $1$.
As in the proof of Prop.\ \ref{earinmi}, for an ear $f$ we write $I_f$ for the set of $i$ with $f\in m_i$. Remember that $I_f$ is an ``interval'', that the number of $I_f$'s is $k$, and that two such $I_f$'s never intersect in consecutive positions (Cor.~\ref{atmost2}). These three facts together with (\ref{m1m2}) imply that the entries $*$ are equal to $2$.
We also notice that $I_f$ can never be the empty set, whence (3) (we check $n=3,4$ separately).
\end{proof}

\begin{lemma}\label{ef1}
Let $c\in\cAp$ be \maxdensen, $e<f$ be ears, and $\ell\notin\{e,\ldots,f\}$ with $m_\ell=\{e,f\}$.
Then either $c=(1,1,1)$ or $f=e+2$.
\end{lemma}
\begin{proof}
We prove that there is no ear $j$ with $e<j<f$. Assume the converse and let $e<j<f$ be an ear.
Without loss of generality, $e<j<f<\ell$.
Then by Prop.\ \ref{gggelll}, $\varphi_i(j)<\varphi_i(e)$ for all $j\le i \le \ell$, and
$\varphi_i(j)<\varphi_i(f)$ for all $\ell\le i \le n$ and $1\le i \le j$. Thus there is no $m_i$ with $j\in m_i$; this contradicts Corollary \ref{cform} (3).
\end{proof}
\begin{corol}\label{ee2}
Let $c\in\cAp$ be \maxdense of length $n>4$.
\begin{enumerate}
\item For every $i$ there exists an ear $e$ with $m_i\in\{ \{e\}, \{e,e+2\}\}$.
\item For every ear $e$ there exists an $i$ with $m_i=\{e,e+2\}$.
\end{enumerate}
\end{corol}
\begin{proof}
(1) follows directly from Lemma \ref{ef1} and Corollary \ref{cform} (2), (3).
Ad (2): There are $k=\frac{n}{2}$ pairs of ears $\{e,e+2\}$ and $k$ sets $m_i$ with $|m_i|>1$.
\end{proof}

\begin{lemma}\label{gcdab}
Let $c\in\cAp$, $i,j,\ell$ be vertices of the corresponding triangulation, and assume that $j,\ell$ are connected by an edge.
Then $$\gcd(\varphi_i(j),\varphi_i(\ell))=1$$ for every $i\notin\{j,\ell\}$.
In particular, if $\varphi_i(j)=\varphi_i(\ell)$ then $$\varphi_i(j)=\varphi_i(\ell)=1.$$
\end{lemma}
\begin{proof}
The numbers $\varphi_i(j),\varphi_i(\ell)$ are coordinates of a root in a root system of a Weyl groupoid (see for example \cite{p-CH09d}).
\end{proof}

\begin{theor}\label{cdense}
Let $c\in\cAp$ be \maxdensen. Then up to cyclic rotations, $c$ is one of
the following sequences:
\[ (1,1,1),\quad (1,2,1,2),\quad (1,3,1,3,1,3),\]
\[ (1,3,1,4,1,3,1,4),\quad (1,3,1,5,1,3,1,5,1,3,1,5). \]
\end{theor}
\begin{proof}
Denote $\cAp':=\{c\in\cAp\mid c_i=1 \text{ for all even } i \}$.
By Cor.\ \ref{cform} we may assume that $c=(*,1,*,1,\ldots)$, so $c\in\cAp'$. The only $\eta$-sequences of length $n\le 4$ are $(0,0)$, $(1,1,1)$, and $(1,2,1,2)$ (up to rotations); these cases are easy to compute. Thus assume that $n>4$ (hence $n\ge 6$ by Corollary \ref{cform}, and $c$ is not fan-shaped). For $a\in\ZZ$ let
\[ \xi(a) := \eta(a) \eta(1). \]
Then we have the rule
\[ \xi(a)\xi(3)\xi(b) = \xi(a-1)\xi(b-1) \]
for all $a,b$, and thus obtain that
\[ \psi : \cAp'\backslash\{(1,1,1)\} \rightarrow \cAp, \quad (c_1,1,c_3,1,\ldots) \mapsto (c_1-2,c_3-2,\ldots) \]
is a bijection (this corresponds to removing all ears at once in the triangulation).

Every element in $\cAp\backslash\{(0,0)\}$ has an ear, thus every element in $$\cAp'\backslash\{(1,1,1),(2,1,2,1)\}$$ has a three somewhere,
say $c=(*,1,3,1,*,\ldots)$.
By Corollary \ref{ee2}
there exists an $i>4$ with $m_i=\{2,4\}$.
If $\varphi_i(1)=a$ and $\varphi_i(5)=b$, then $\varphi_i(3)=a+b$, $\varphi_i(2)=2a+b$, and $\varphi_i(2)=a+2b$.
But $\varphi_i(2)=\varphi_i(4)$ or $2a+b=a+2b$. Then $a=b$ which is only possible if $a=b=1$ by Lemma \ref{gcdab}. Hence in the triangulation we have a triangle connecting $1$, $5$ and $i$, since by Remark \ref{ci1}, $\varphi_i(j)=1$ if and only if $i$ and $j$ are connected by an edge.

One computes $\varphi_i(6)=c_5-3$. But $2,4\in m_i$, so $3=\varphi_i(2)>\varphi_i(6)=c_5-3$. Thus we see that $2<c_5\le 5$.

If $c_5=3$ then $(1,3,1,3)$ is a subsequence of $c$. This implies $c=(3,1,3,1,3,1)$ since $(1,1,1)$ is the only element of $\cAp$ with two neighboring ones. If $c_5>3$, then we have at least $5$ triangles and thus $n>6$; so from now on let $n\ge 8$.

Assume $c_5=4$. Notice first that $m_1=\{4,6\}$. Indeed, let $m_j=\{4,6\}$ for some $j$, and write $a=\varphi_j(1)$, $b=\varphi_j(7)$.
Then after excluding $j\in\{2,3,4,5,6\}$ (which is easy) we get $$3a+2b=\varphi_j(4)=\varphi_j(6)=a+2b,$$ thus $a=0$ and we conclude that $j=1$.
Now if $c_7$ was greater than $3$, then $\varphi_1(8)\ge \varphi_1(4)$ and then $m_1$ would contain $8$, this contradicts $m_1=\{4,6\}$.
Further, $c_7=2$ is impossible because $(1,2,1,2)$ and $(2,1,2,1)$ are the only elements of $\cAp$ containing $(1,2,1)$.
Thus $c_7=3$ and $c=(c_1,1,3,1,4,1,3,\ldots)$. This implies $c=(4,1,3,1,4,1,3,1)$ since $(1,2,1,2)$ and $(2,1,2,1)$ are the only elements of $\cAp$ containing $(1,2,1)$, and $(1,2,1)$ is a subsequence of $\psi(c)$.

The last case is $c_5=5$.
Again, let $m_j=\{4,6\}$ for some $j$, and write $a=\varphi_j(1)$, $b=\varphi_j(i)$. This time we exclude $j\in\{2,3,4,5,6,7\}$ and we get
\[ 3a+2b = \varphi_j(4) = \varphi_j(6) = 2a+3b, \]
or $a=b$ and hence $a=b=1$ since $1$ and $i$ are connected by an edge. Thus $\varphi_j(4)=\varphi_j(6)=5$. Now if $c_7$ was greater than $3$, then $\varphi_j(8)$ would be greater than $5$, contradicting $8\notin m_j$.
This shows that $c=(c_1,1,3,1,5,1,3,1,\ldots)$ and thus that $c$ is periodic with period $(1,3,1,5)$. But then $\psi(c)=(3,1,3,1,\ldots)$ which is only possible if $c=(5,1,3,1,5,1,3,1,5,1,3,1)$ (compare \cite[Prop.\ 3.11]{p-CH09d}).
\end{proof}

\section{Applications}

\subsection{Weyl groupoids and arrangements}

We briefly recall the notions of Weyl groupoids and crystallographic arrangements, see \cite{p-CMa-13} for the most general definitions, and \cite{p-CH09a}, \cite{p-C10} for the original definitions in the finite case. Finite simplicial arrangements were introduced in \cite{a-Melchi41}.

\begin{defin}[compare {\cite{p-CMa-13}}]
Let $V=\RR^r$ be a finite dimensional real vector space.
Let $\Ac$ be a (possibly infinite) central arrangement in $V$, i.e.\ a set of linear hyperplanes in $V$, and $\Kc (\Ac)$ be the set of connected components of $V\backslash \bigcup_{H\in\Ac} H$.
Let $\emptyset\ne T\subseteq V$ be an open connected convex cone.
We call the pair $(\Ac,T)$ a \emph{thin simplicial arrangement} if:
\begin{enumerate}
\item[(S1)] $\Kc_T(\Ac):= \{ K\in \Kc (\Ac)\mid K\subseteq T\}$ consists of open simplicial cones called \emph{chambers}.
\item[(S2)] $H\cap T\ne \emptyset$ for all $H\in\Ac$.
\item[(S3)] For all $0\ne v\in T$ there exists a neighborhood $U$ of $v$ such that $\{H\in\Ac\mid H\cap U\ne\emptyset \}$ is finite.
\item[(S4)] The \emph{walls} of each chamber are elements of $\Ac$.
\end{enumerate}
\end{defin}

If $T=V$, then $\Ac$ is finite, and if $T$ is a half space, then $\Ac$ is \emph{affine}.
In this note we will omit the word ``thin'' (which corresponds to Axiom (S4)) since all considered arrangements are thin when they are simplicial.

\begin{defin}[compare {\cite{p-CMa-13}}]
A \emph{crystallographic arrangement} is a triple $(\Ac,T,R)$ where
$\Ac$ is a simplicial arrangement with Tits cone $T$, and
$R\subseteq V^*\backslash\{0\}$ such that:
\begin{enumerate}
\item[(CA)]
$\bullet$ \: $R\cap \langle\alpha\rangle_\RR = \{\pm\alpha\}$ for all $\alpha\in R$.
\\ \noindent$\bullet$ \:  $\Ac=\{\alpha^\perp\mid \alpha\in R\}$.
\\ \noindent$\bullet$ \:  For each $K\in\Kc_T(\Ac)$, let $\alpha_1,\ldots,\alpha_r\in R$ be such that $K$ is the dual cone of $\langle\alpha_1,\ldots,\alpha_r\rangle_{>0}$. Then
\label{cra:cryst} 
\[ R \subseteq \pm \sum_{i=1}^r \NN_0\alpha_i.\]
\end{enumerate}
\end{defin}

We ommit the definitions of Cartan schemes and Weyl groupoids because they are quite long and will not be needed in the sequel (see \cite{p-CH09a} for details).
It suffices to mention that connected simply connected Cartan schemes correspond to crystallographic arrangements.
Each chamber $K\in\Kc_T(\Ac)$ provides a unique basis as in (CA); with respect to this basis, the set $R$ becomes a subset $R^K\subseteq\ZZ^r$ which we call the \emph{root system} of $\Ac$ at $K$.

\subsection{The root posets of finite Weyl groupoids of rank two}
Finite crystallographic arrangements of rank two are in one-to-one correspondence with $\eta$-sequences (see for example \cite{p-CH09d}).
They are obtained in the following way: Start with a triangulation of a convex polygon by non-intersecting diagonals. Choose two neighboring vertices $i,i+1$, and write $(1,0)$ at vertex $i$ and $(0,1)$ at vertex $i+1$. Then if two vertices of a triangle have labels $\alpha$ and $\beta$, then write $\alpha+\beta$ at the third vertex. (This is exactly the procedure to get the numbers $\varphi_i(j)$ performed at two vertices simultaneously.) The set of all labels $R_+\subseteq\ZZ^2$ is the set of positive roots corresponding to the ``chamber'' $(i,i+1)$:
\[ R_+ = \{(\varphi_i(j),\varphi_{i+1}(j)) \mid j=1,\ldots,n\}. \]
The set $R_+$ becomes a poset via
\[ (a,b)\le(c,d) \quad:\Longleftrightarrow\quad a\le c \text{ and } b\le d. \]

Theorem \ref{cdense} yields the following result about the root posets $R_+$.

\begin{corol}\label{maxelt_in_poset}
Let $c\in\cAp$ and let $\Ac$ be the crystallographic arrangement associated to $c$.
Then there exists a chamber $K$ such that the root system $R^K$ at $K$ contains a unique maximal element $(x,y)\in\ZZ^2$.\\
Moreover, if $c$ is not dense then
\[ x>u \text{ and } y>v \quad\text{for all}\quad (x,y)\ne(u,v)\in R^K. \]
\end{corol}
\begin{proof}
The claim is easy to check if $c$ is dense. Assume now that $c$ is not dense. Then there exists a label $i$ such that $|m_i|=|m_{i+1}|=1$. But then $m_i=\{x\}$ and $m_{i+1}=\{y\}$ for some $x,y\in\NN$. The root $(x,y)$ is thus maximal by definition of the $m_i$'s.
\end{proof}

\subsection{Affine Weyl groupoids of rank three}

\begin{defin}
Let $(\Ac,T)$ be an affine simplicial arrangement of rank three in $V=\RR^3$. We say that $(\Ac,T)$ is of \emph{imaginary type $A_1^{(1)}$} if there exists a (finite) crystallographic arrangement $(\hat\Ac,\hat T,\hat R)$ of rank two, an \emph{imaginary root} $\imr\in V^*$, and an embedding $\iota : \hat R\rightarrow V^*$ such that
\[ \Ac = \{\alpha^\perp\mid \alpha\in \iota(\hat R)+\ZZ\imr\}. \]
\end{defin}

Fig.\ \ref{fig1315} is an example for an affine simplicial arrangement of imaginary type $A_1^{(1)}$.
We may use Thm.\ \ref{cdense} to prove the following:

\begin{theor}\label{mainthm}
If $(\Ac,T)$ is of imaginary type $A_1^{(1)}$, then the characteristic sequence of $\hat R$ is
\[ (1,1,1),\: (1,2,1,2),\: (1,3,1,3,1,3),\: \text{ or } \: (1,3,1,5,1,3,1,5,1,3,1,5). \]
In other words, either there exists a set $R$ such that $(\Ac,T,R)$ is a crystallographic arrangement where $R$ is the root system of an affine Weyl group of type $A$, $B$, $G$, or $(\Ac,T)$ is the exceptional arrangement presented in Fig.\ \ref{fig1315}.
The arrangement from Fig.\ \ref{fig1315} is not crystallographic.
\end{theor}

\begin{remar}
Notice that the finite Weyl groupoid of rank two corresponding to the sequence $(1,3,1,5,1,3,1,5,1,3,1,5)$ already appeared to be the largest root system in the classification of finite dimensional Nichols algebras of diagonal type of rank two (see \cite{a-Heck08a}).
\end{remar}

\subsection{Proof of Theorem \ref{mainthm}}

In this section, let $\hat R\subseteq\ZZ^2$ be a finite crystallographic arrangement of rank two to the $\eta$-sequence $c=(c_1,\ldots,c_n)$ and
\begin{eqnarray*}
R &=& \{ (a,b,d) \mid d\in\ZZ,\:\: (a,b)\in \hat R \},\\
\Ac &=& \{ \alpha^\perp \mid \alpha\in R \}.
\end{eqnarray*}
(The imaginary root will be $\imr=(0,0,1)$.)
We will write $\Rh:=\{(a,b,0)\mid (a,b) \in \hat R\}$.

\begin{figure}
\begin{center}
\includegraphics[width=\textwidth]{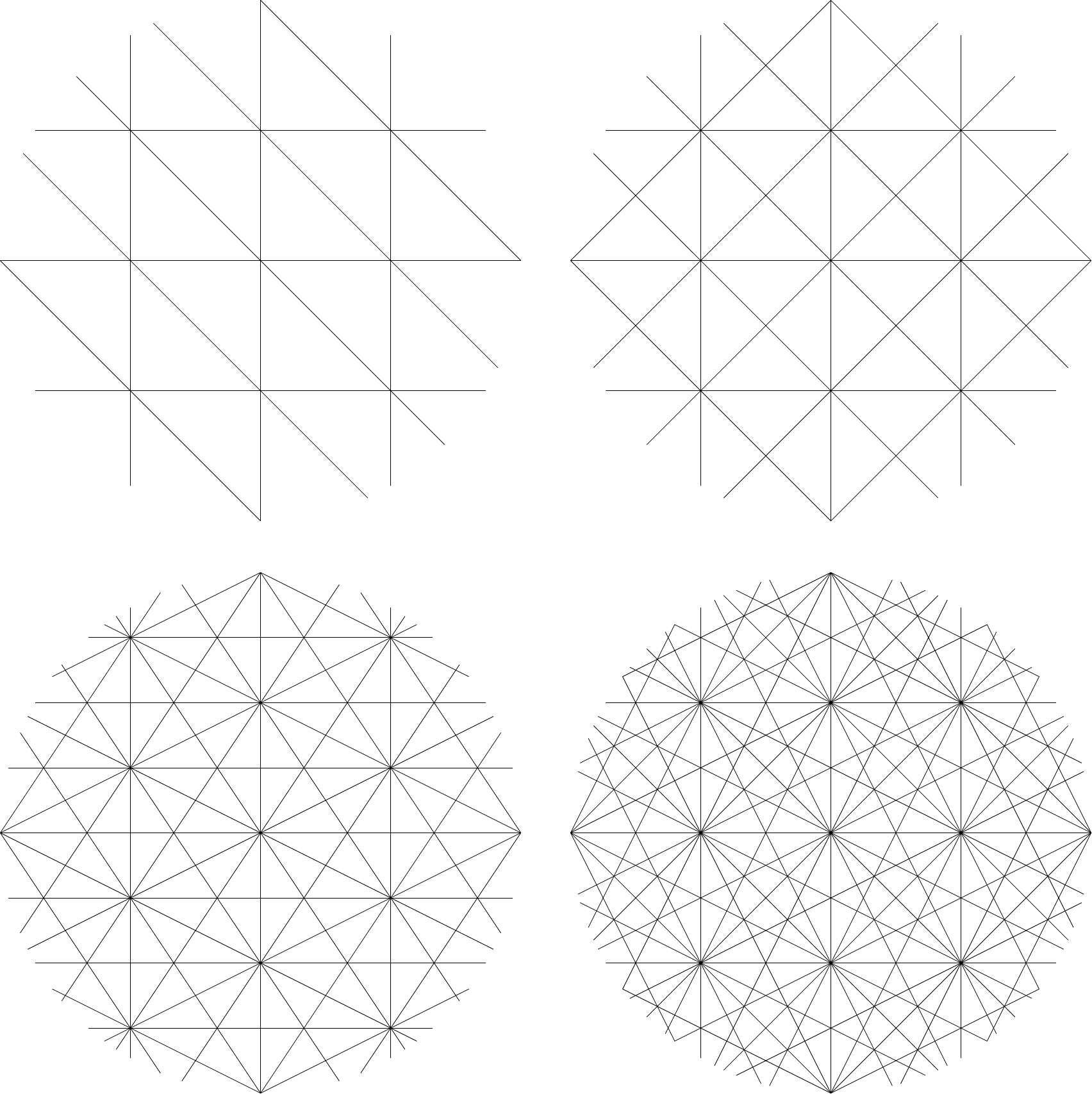}
\end{center}
\caption{Affine arrangements for the sequences $(1,1,1)$, $(1,2,1,2)$, $(1,3,1,3,1,3)$, $(1,3,1,4,1,3,1,4)$\label{denseaff}}
\end{figure}

\begin{propo}
If $\Ac$ is a simplicial arrangement, then $c$ is dense (see Def.\ \ref{defdense}).
\end{propo}
\begin{proof}
A necessary condition to obtain a simplicial arrangement is that all chambers adjacent to $\langle \imr\rangle_{\ge 0}$ are open simplicial cones.
Let $\alpha\in\Rh$ and define
\[ S:=\{\alpha^\perp\cap \gamma^\perp \mid \gamma \in R\backslash \Rh\} \]
to be the set of (projective) intersection points of hyperplanes lying on $\alpha^\perp$.
Choose the basis corresponding to a chamber adjacent to $\langle \imr\rangle_{\ge 0}$ and having $\alpha^\perp$ as a wall. With respect to this basis, $\alpha=(1,0,0)$ and $\imr=(0,0,1)$.
Then for $\gamma=(a,b,d)\in R$ we have $\alpha^\perp\cap\gamma^\perp=\langle (0,-\frac{d}{b},1)\rangle$.

Let $\{p_1,p_2\}\subseteq S\backslash \{\langle(0,0,1)\rangle\}$
be the two points of $S$ nearest to $(0,0,1)$ when projected to the affine hyperplane $H:=\{(a,b,1)\mid a,b\in\RR\}$.
Let $\beta_1=(a_1,b_1,d_1)$, $\beta_2=(a_2,b_2,d_2)\in R\backslash \Rh$ be such that
$\alpha^\perp\cap \beta_i^\perp=p_i$, $i=1,2$.
Since $(a_1,b_1,d_1)\in R$ implies $(a_1,b_1,1)\in R$ and since $p_1\cap H$ is nearest to $(0,0,1)$,
$d_1=1$; for the same reason, $d_2=1$. Both points have the same distance to $(0,0,1)$, thus $\frac{1}{b_1^2}=\frac{1}{b_2^2}$.
The number of roots $\beta$ such that $\alpha^\perp\cap\beta^\perp=p_1$ is thus
\[ m_\alpha := |\{(a,b,1)\in R \mid b=b_1\}|, \]
where $b_1 = \max \{b\mid (a,b,1)\in R\}$. Notice that this number depends on the coordinates and thus on the chosen basis; however it does not depend on the chosen chamber as above. We obtain a well defined number $m_\alpha$.

Now choose an ordering $\alpha_1,\ldots,\alpha_n,-\alpha_1,\ldots,-\alpha_n$ of the roots of $\Rh$ in such a way that two consecutive roots define walls of a chamber. These chambers can only be simplicial cones if there are no two consecutive $1$'s in the series
\[ m_{\alpha_1},m_{\alpha_2},\ldots,m_{\alpha_n},m_{-\alpha_1},\ldots,m_{-\alpha_n}. \]
It is easy to see that these numbers are exactly the $|m_i|$'s attached to the sequence $c=(c_1,\ldots,c_n)$ as in Def.\ \ref{defmis}.
Thus simpliciality of the arrangement defined by $R$ implies that $c$ is \maxdensen.
\end{proof}

Thus by Thm.\ \ref{cdense}, there are only five possible root systems $\hat R$ (up to isomorphisms). It is easy to see that the $\eta$-sequence $(1,3,1,4,1,3,1,4)$ does not define an affine simplicial arrangement (see Fig.\ \ref{denseaff}). The sequences $(1,1,1)$, $(1,2,1,2)$, and $(1,3,1,3,1,3)$ yield the classical affine arrangements of types $A$, $B$, and $G$.
The following proposition completes the proof of Thm.\ \ref{mainthm}.

\begin{propo}
The set
\[ R = \{ (x,y,z) \mid c\in\ZZ,\:\: (x,y)\in \hat R \}, \]
where $\hat R$ is the crystallographic arrangement
corresponding to the $\eta$-sequence $(1,3,1,5,1,3,1,5,1,3,1,5)$
is not an affine crystallographic arrangement.
\end{propo}
\begin{proof}
There are $108$ chambers in the fundamental domain. Fixing an adequate chamber $K$, there are $84$ chambers in the fundamental domain with base change of determinant $\pm 1$, and $24$ chambers in the fundamental domain with base change of determinant $\pm 1/2$.
\end{proof}

\begin{figure}
\includegraphics[width=0.48\textwidth]{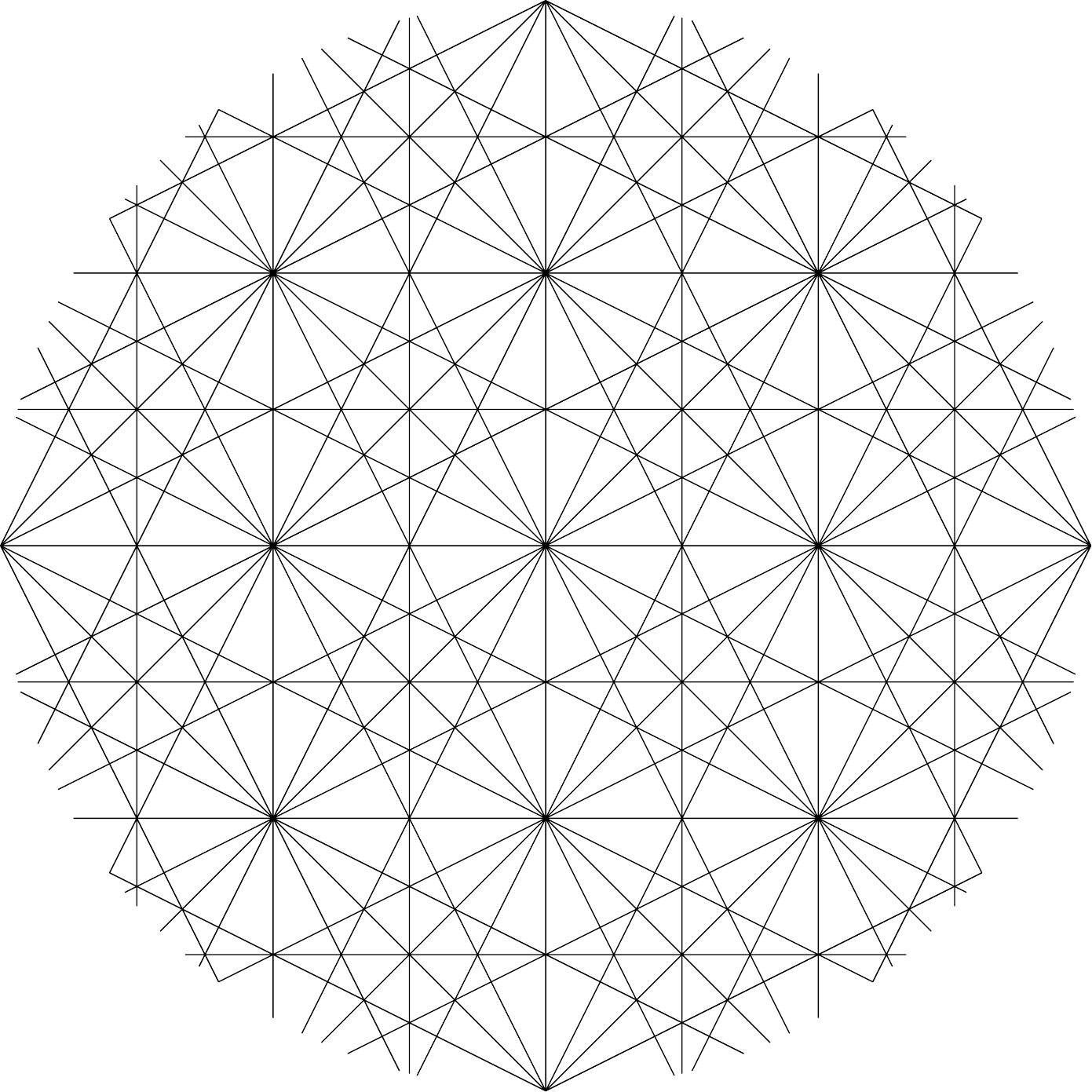}
\includegraphics[width=0.48\textwidth]{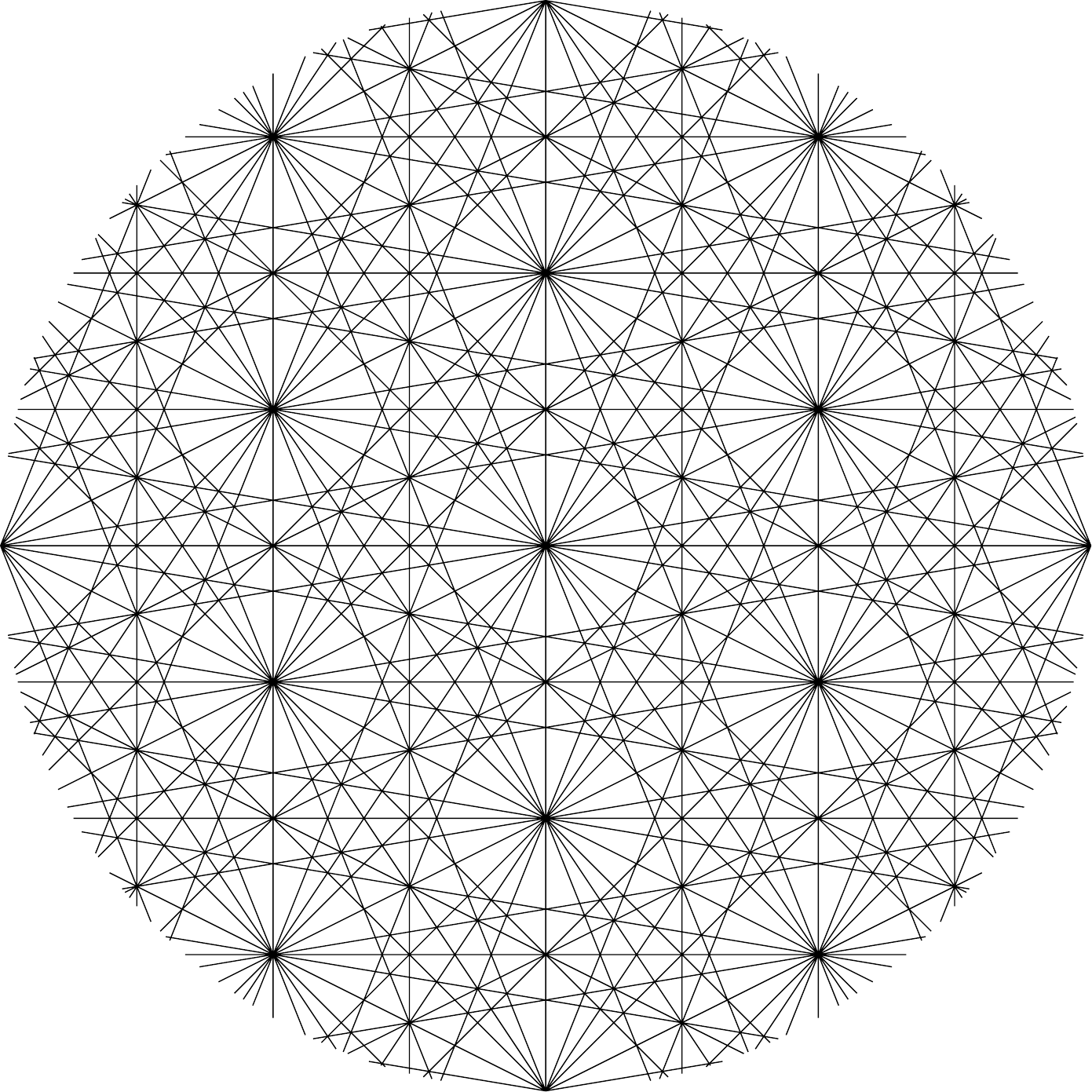}
\caption{Crystallographic arrangements based on the sequences $(1,3,1,4,1,3,1,4),(1,3,1,5,1,3,1,5,1,3,1,5)$\label{cry1415}}
\end{figure}

There exist crystallographic arrangements based on the sequences $(1,3,1,4,1,3,1,4)$ and $(1,3,1,5,1,3,1,5,1,3,1,5)$, i.e.\ adding further parallels yields a simplicial arrangement with Axiom (CA), see Figure \ref{cry1415}.
At a meeting in Oberwolfach, Bernhard M\"uhlherr and the author conjectured:
\begin{conje}\label{affwgconj}
Let $(\Ac,T,R)$ be an affine crystallographic arrangement of rank three in $V=\RR^3$, and let
$(\hat\Ac,\hat T,\hat R)$ be a (finite) crystallographic arrangement of rank two such that
\[ R \subseteq \iota(\hat R)+\ZZ\imr \]
for some $\imr\in V^*$ and $\iota : \hat R\rightarrow V^*$.
Then the characteristic sequence of $\hat R$ is one of those occurring in the classification of arithmetic root systems of rank two (see \cite{a-Heck08a}):
\[ (1,1,1),(1,2,1,2),(1,2,2,1,3),(1,2,2,2,1,4),(1,3,1,3,1,3), \]
\[ (1,2,3,1,3,2,1,5),(1,3,1,4,1,3,1,4),(1,2,3,2,1,4,1,4), \]
\[ (1,3,1,5,1,3,1,5,1,3,1,5). \]
\end{conje}


\begin{thebibliography}{10}

\bibitem{p-AHS-08}
N.~Andruskiewitsch, I.~Heckenberger, and H.~J. Schneider, \emph{The {N}ichols
  algebra of a semisimple {Y}etter-{D}rinfeld module}, Amer. J. Math.
  \textbf{132} (2010), no.~6, 1493--1547.

\bibitem{Arm2006}
D.~Armstrong, \emph{Generalized noncrossing partitions and combinatorics of
  {C}oxeter groups}, Mem. Amer. Math. Soc. \textbf{202} (2006), no.~949.

\bibitem{pCfC06}
P.~Caldero and F.~Chapoton, \emph{Cluster algebras as {H}all algebras of quiver
  representations}, Comment. Math. Helv. \textbf{81} (2006), no.~3, 595--616.

\bibitem{jChC73}
J.~H. Conway and H.~S.~M. Coxeter, \emph{Triangulated polygons and frieze
  patterns}, Math. Gaz. \textbf{57} (1973), 87--94 and 175--183.

\bibitem{p-C10}
M.~Cuntz, \emph{Crystallographic arrangements: Weyl groupoids and simplicial
  arrangements}, Bull. London Math. Soc. \textbf{43} (2011), no.~4, 734--744.

\bibitem{p-CH09b}
M.~Cuntz and I.~Heckenberger, \emph{Weyl groupoids of rank two and continued
  fractions}, Algebra \& Number Theory \textbf{3} (2009), 317--340.

\bibitem{p-CH09a}
\bysame, \emph{Weyl groupoids with at most three objects}, J.~Pure
  Appl.~Algebra \textbf{213} (2009), no.~6, 1112--1128.

\bibitem{p-CH09d}
\bysame, \emph{Reflection groupoids of rank two and cluster algebras of type
  ${A}$}, J. Combin. Theory Ser. A \textbf{118} (2011), no.~4, 1350--1363.

\bibitem{p-CMa-13}
M.~Cuntz, B.~M\"uhlherr, and C.~Weigel, \emph{On the {T}its cone of the {W}eyl
  groupoid}, in preparation (2014).

\bibitem{MR1303141}
R.~Guy and R.~Woodrow (eds.), \emph{The lighter side of mathematics}, MAA
  Spectrum, Washington, DC, Mathematical Association of America, 1994.

\bibitem{p-H-06}
I.~Heckenberger, \emph{The {W}eyl groupoid of a {N}ichols algebra of diagonal
  type}, Invent.~Math. \textbf{164} (2006), no.~1, 175--188.

\bibitem{a-Heck08a}
\bysame, \emph{Rank 2 {N}ichols algebras with finite arithmetic root system},
  Algebr. Represent. Theory \textbf{11} (2008), no.~2, 115--132.

\bibitem{a-Melchi41}
E.~Melchior, \emph{\"{U}ber {V}ielseite der projektiven {E}bene}, Deutsche
  Math. \textbf{5} (1941), 461--475.

\end{thebibliography}

\def\cprime{$'$}
\providecommand{\bysame}{\leavevmode\hbox to3em{\hrulefill}\thinspace}
\providecommand{\MR}{\relax\ifhmode\unskip\space\fi MR }
\providecommand{\MRhref}[2]{%
  \href{http://www.ams.org/mathscinet-getitem?mr=#1}{#2}
}
\providecommand{\href}[2]{#2}

\end{document}